\documentclass[11pt]{amsart}
\usepackage{graphicx,amscd,color}
\usepackage{a4wide}

\theoremstyle{plain}
\newtheorem{theorem}{Theorem}[section]

\newtheorem{lemma}[theorem]{Lemma}

\theoremstyle{remark}

\newtheorem{claim}{Claim}

\begin{document}

\title [On weak reducing disks and disk surgery]
{On weak reducing disks and disk surgery}

\author[J. H. Lee]{Jung Hoon Lee}
\address{Department of Mathematics and Institute of Pure and Applied Mathematics,
Chonbuk National University, Jeonju 54896, Korea}
\email{junghoon@jbnu.ac.kr}


\begin{abstract}
Let $K$ be an unknot in $8$-bridge position in the $3$-sphere.
We give an example of a pair of weak reducing disks $D_1$ and $D_2$ for $K$ such that
both disks obtained from $D_i$ ($i = 1, 2$) by a surgery along any outermost disk in $D_{3-i}$,
cut off by an outermost arc of $D_i \cap D_{3-i}$ in $D_{3-i}$, are not weak reducing disks,
i.e. the property of weak reducibility of compressing disks is not preserved by a disk surgery.
\end{abstract}

\maketitle

\section{Introduction}\label{sec1}

A closed $3$-manifold $M$ admits a Heegaard splitting $V \cup_S W$,
a decomposition of $M$ into two handlebodies $V$ and $W$.
The notion of Heegaard splitting extends to a bridge splitting $(V, V \cap K) \cup_S (W, W \cap K)$,
where $K$ is a knot (or link) in $M$.

A {\em {disk complex}} $\mathcal{D}(V)$ (resp. $\mathcal{D}(V - K)$)
is a simplicial complex defined as follows.
(A disk complex $\mathcal{D}(W)$ (resp. $\mathcal{D}(W - K)$) is defined similarly.)

\begin{itemize}
\item Vertices of $\mathcal{D}(V)$ (resp. $\mathcal{D}(V - K)$) are
isotopy classes of compressing disks for $S$ (resp. $S - K$) in $V$ (resp. $V - K$).
\item A collection of $k+1$ vertices forms a $k$-simplex
if there are representatives for each that are pairwise disjoint.
\end{itemize}

There have been many works related to disk complexes to understand the topology of $3$-manifolds,
e.g. the Hempel distance \cite{Hempel}, \cite{Scharlemann-Tomova},
\cite{Ido-Jang-Kobayashi}, \cite{Qiu-Zou-Guo}, \cite{Johnson-Moriah},
a geometric structure of a disk complex \cite{Masur-Schleimer}, \cite{Li},
a subcomplex of a disk complex \cite{Cho-Koda}.

Subcomplexes of a disk complex such as primitive disk complex, weak reducing disk complex are important
since they are related to lower genus (or bridge number) splitting,
hence to a mapping class group (the Goeritz group) \cite{Cho-Koda},
or topological minimality \cite{E}.

We consider weak reducing disks for an unknot in bridge position in $S^3$.
A $2$-bridge position of an unknot admits no weak reducing disk.
The weak reducing disk complex for an unknot in $3$-bridge position is contractible \cite{Kwon-Lee}.
But for $n \ge 4$, it is not known
whether the weak reducing disk complex for an unknot in $n$-bridge position is even connected or not.
When two disks with nonempty intersection are given,
typically we use standard disk surgery to construct a path between them in the disk complex.
In this paper, we give an example such that
a disk surgery for a pair of weak reducing disks does not yield any weak reducing disk.

\begin{theorem}\label{thm1}
Let $K$ be an unknot in $8$-bridge position in $S^3$.
There is a pair of weak reducing disks $D_1$ and $D_2$ for $K$ such that
both disks obtained from $D_i$ ($i = 1, 2$) by a surgery along any outermost disk in $D_{3-i}$,
cut off by an outermost arc of $D_i \cap D_{3-i}$ in $D_{3-i}$, are not weak reducing disks.
\end{theorem}

%We remark that there is a similar example of
%a pair of primitive disks for a genus three Heegaard splitting of $S^3$ [Cho].
There is still a possibility that the weak reducing disk complex is connected
for an $n$-bridge ($n \ge 4$) position of the unknot, but
it would be difficult to find a disk (in a path in the weak reducing disk complex)
that is not obtained by a disk surgery.

\section{Weak reducing disks and disk surgery}\label{sec2}

Let $S^3 = V \cup_S W$ be a decomposition of $S^3$ into two $3$-balls $V$ and $W$
with the common boundary sphere $S$.
Let $K$ be a knot in $S^3$ such that
$V \cap K$ and $W \cap K$ are collections of $n$ boundary parallel arcs.
Then $(V, V \cap K) \cup_S (W, W \cap K)$ is called an {\em {$n$-bridge splitting}} of $(S^3, K)$, and
$S - K$ is the {\em {$n$-bridge sphere}}.
We say that $K$ is in {\em {$n$-bridge position}} with respect to $S$.
Each arc of $V \cap K$ and $W \cap K$ is called a {\em {bridge}}.

A {\em {bridge disk}} $\Delta$ in $V$ is a disk such that
$\partial \Delta$ is the endpoint union of two arcs $a$ and $b$,
where $b$ is a bridge and $a$ is an arc in $S$, and
$\Delta \cap K = b$.
In other words, $\Delta$ is a boundary parallelism disk.
A disk $D$ in $V - K$ is a {\em {compressing disk}} if $\partial D$ is essential in $S - K$.
For a bridge disk $\Delta$, the frontier of a neighborhood of $\Delta$ in $V$
is called a {\em {cap}} over $\Delta$.
All the above notions are defined similarly in $W$.
A compressing disk $D$ in $V - K$ is a {\em {weak reducing disk}} if
there is a compressing disk $E$ in $W - K$ such that $D \cap E = \emptyset$, and
$(D, E)$ is called a weak reducing pair.

Suppose that two disks $D_1$ and $D_2$ intersect in a collection of arcs.
Let $\alpha$ be an outermost arc of $D_1 \cap D_2$ in $D_2$ and
$\Delta$ be the corresponding outermost disk in $D_2$ cut off by $\alpha$.
The arc $\alpha$ cuts $D_1$ into two disks $D'_1$ and $D''_1$.
We call $D'_1 \cup \Delta$ and $D''_1 \cup \Delta$
the disks obtained by a {\em {surgery of $D_1$ along $\Delta$}}.

Let $\Delta_1, \ldots, \Delta_n$ be a collection of $n$ pairwise disjoint bridge disks in $W$.
Let $a_i = \Delta_i \cap S$.
The complement of $K$ in $W$ is a genus $n$ handlebody, so
$\pi_1 (W - K)$ is a free group on $n$ generators.
A simple closed curve $\gamma$ in $S - K$ represents an element of $\pi_1 (W - K)$ and
it is written as a word $w$ on $n$ generators:
each time $\gamma$ passes through an $a_i$, the same generator $a^{\pm 1}_i$ is given to $w$.

\begin{lemma}\label{lem1}
If a simple closed curve $\gamma$ in $S - K$ bounds a disk in $W - K$,
then the word $w$ of $\gamma$ is freely reduced to an empty word.
\end{lemma}

\begin{proof}
Suppose that $w$ is not freely reduced to an empty word.
Then $w$ is not a trivial element in $\pi_1 (W - K)$.
It contradicts that $\gamma$ bounds a disk in $W - K$.
\end{proof}

\section{Proof of Theorem \ref{thm1}}\label{sec3}

Let $K$ be an unknot in $8$-bridge position with respect to $S$.
It is well known that an $n$-bridge position of the unknot
is unique for every natural number $n$ \cite{Otal}.
We fix a collection of $16$ bridge disks $\Delta_i$ ($i = 1, \ldots, 16$) for $K$ such that
the union of all the arcs $a_i = \Delta_i \cap S$ ($i = 1, \ldots, 16$) is a simple closed curve.
Here, $\Delta_i$ is in $V$ and $a_i$ is colored red for an odd $i$, and
$\Delta_i$ is in $W$ and $a_i$ is colored blue for an even $i$.
See Figure \ref{fig1}.

\begin{figure}[ht!]
\begin{center}
\includegraphics[width=11cm]{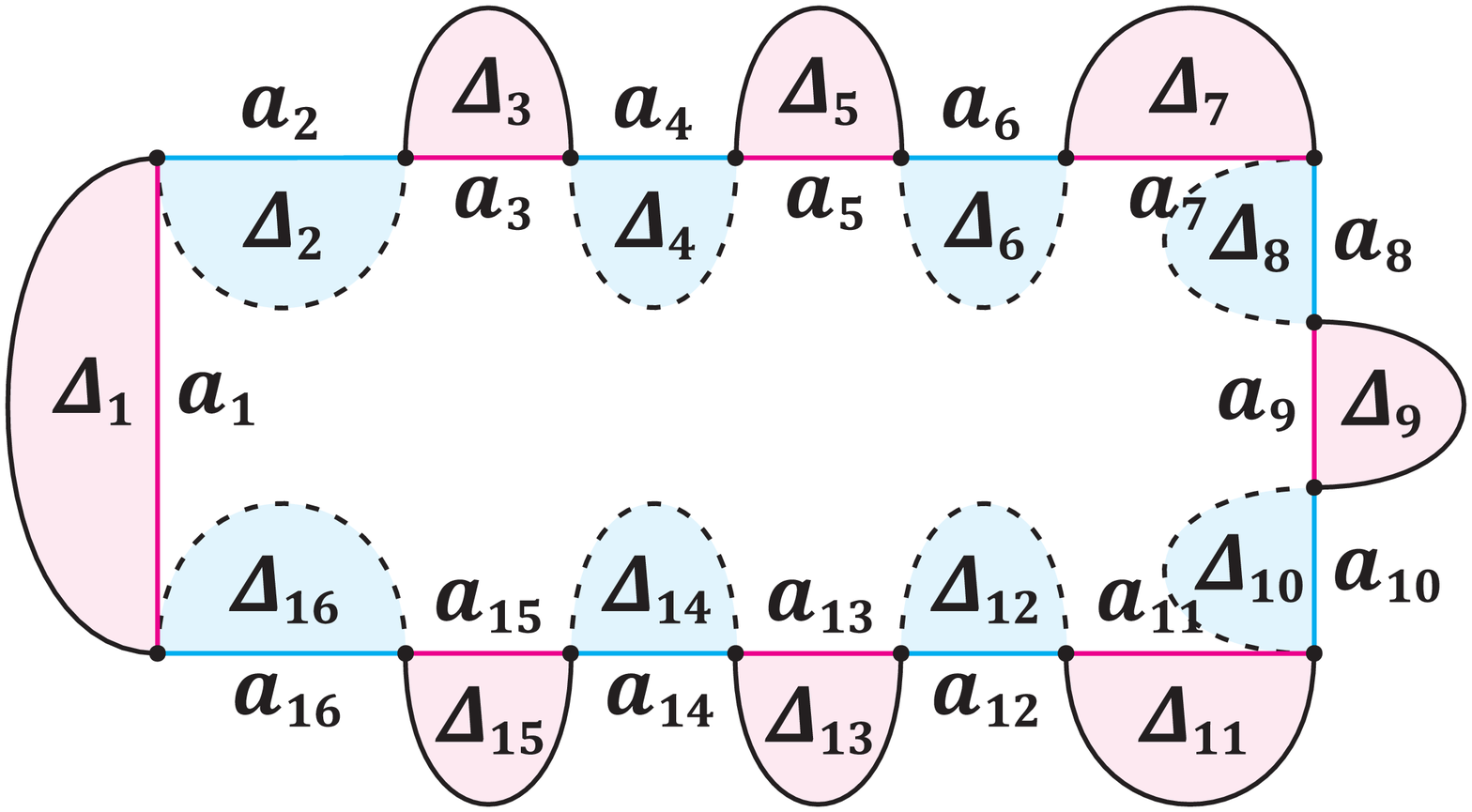}
\caption{$\Delta_i$ and $a_i$ ($i = 1, \ldots, 16$).}\label{fig1}
\end{center}
\end{figure}

We draw five disjoint blue loops $l_i$ ($i = 1, \ldots, 5$)
in $S - \bigcup^8_{i=1} a_{2i}$ as in Figure \ref{fig2}.
Each $l_i$ bounds a compressing disk $E_i$ in $W - K$.
A component of the complement of the union of $l_i$'s in $S - K$ is either
a pair of pants or a $4$-punctured disk.
But if we cut a $4$-punctured disk along two blue arcs in it, then we get a pair of pants.
So we may regard that the blue arcs and loops give a pants decomposition $\{ P_i \}^6_{i=1}$.
Let $P_1$ be the pair of pants containing $a_1$.
For our convenience, let $X = (\bigcup^8_{i=1} \Delta_{2i}) \cup (\bigcup^5_{i=1} E_i$) and
$x = (\bigcup^8_{i=1} a_{2i}) \cup (\bigcup^5_{i=1} l_i)$,
i.e. $X$ is the union of bridge disks in $W$ and compressing disks in $W - K$, and
$x$ is the union of blue arcs and loops.

\begin{figure}[ht!]
\begin{center}
\includegraphics[width=10cm]{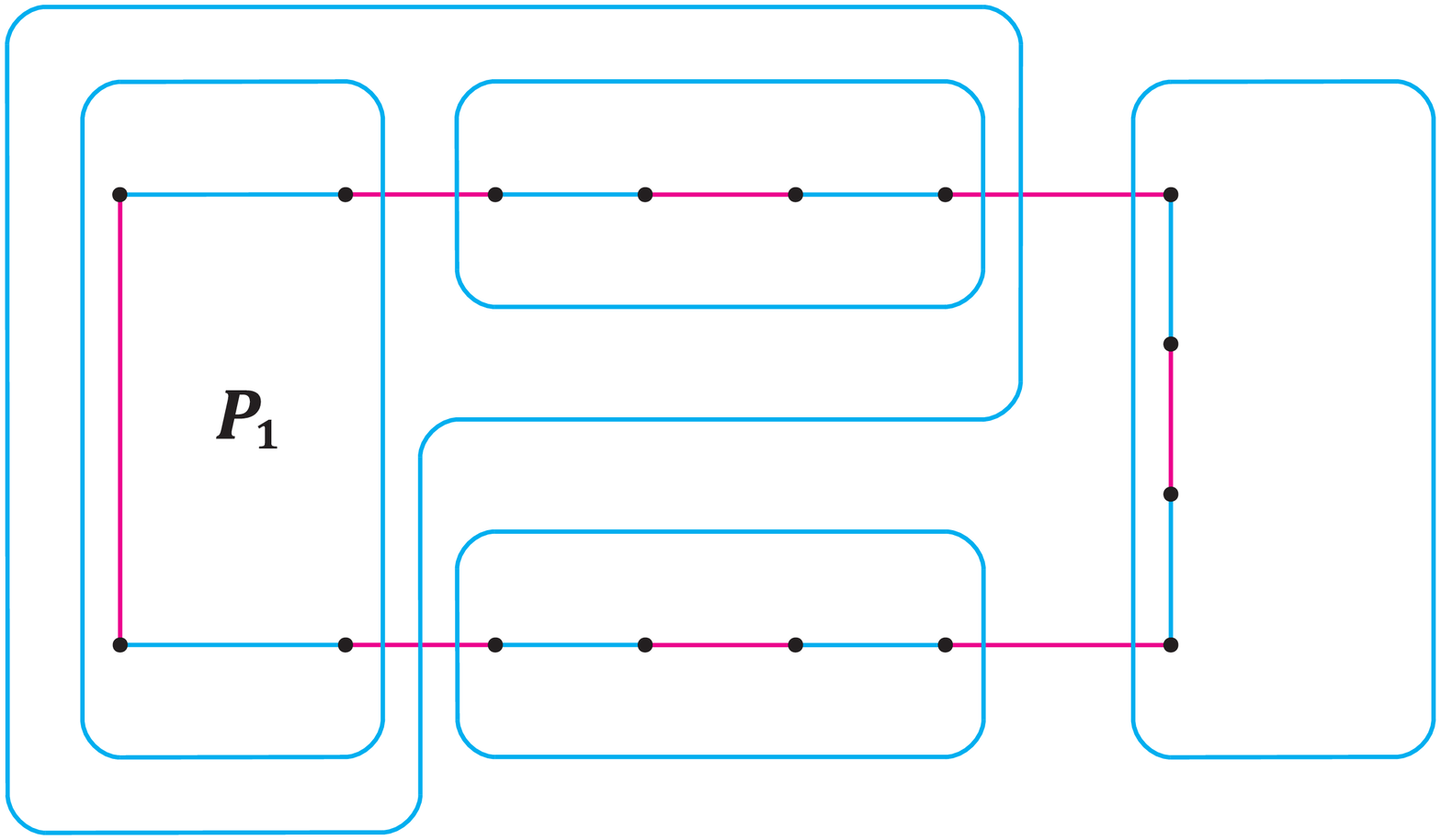}
\caption{A pants decomposition.}\label{fig2}
\end{center}
\end{figure}

Let $D_1$ be a cap over $\Delta_1$.
It is obvious that $D_1$ is a weak reducing disk.
The other disk $D_2$ will be obtained from disks simpler than $D_2$ by band sums.
We can see that the three curves in Figure \ref{fig3}
bound compressing disks, denoted by $C_1, C_2, C_3$, in $V - K$ respectively.
The disk $D'_2$ is a band sum of $C_1$ and $C_2$ as in Figure \ref{fig4}.
Let $d'_2 = \partial D'_2$.
It is essential in $S - K$.

\begin{figure}[ht!]
\begin{center}
\includegraphics[width=12cm]{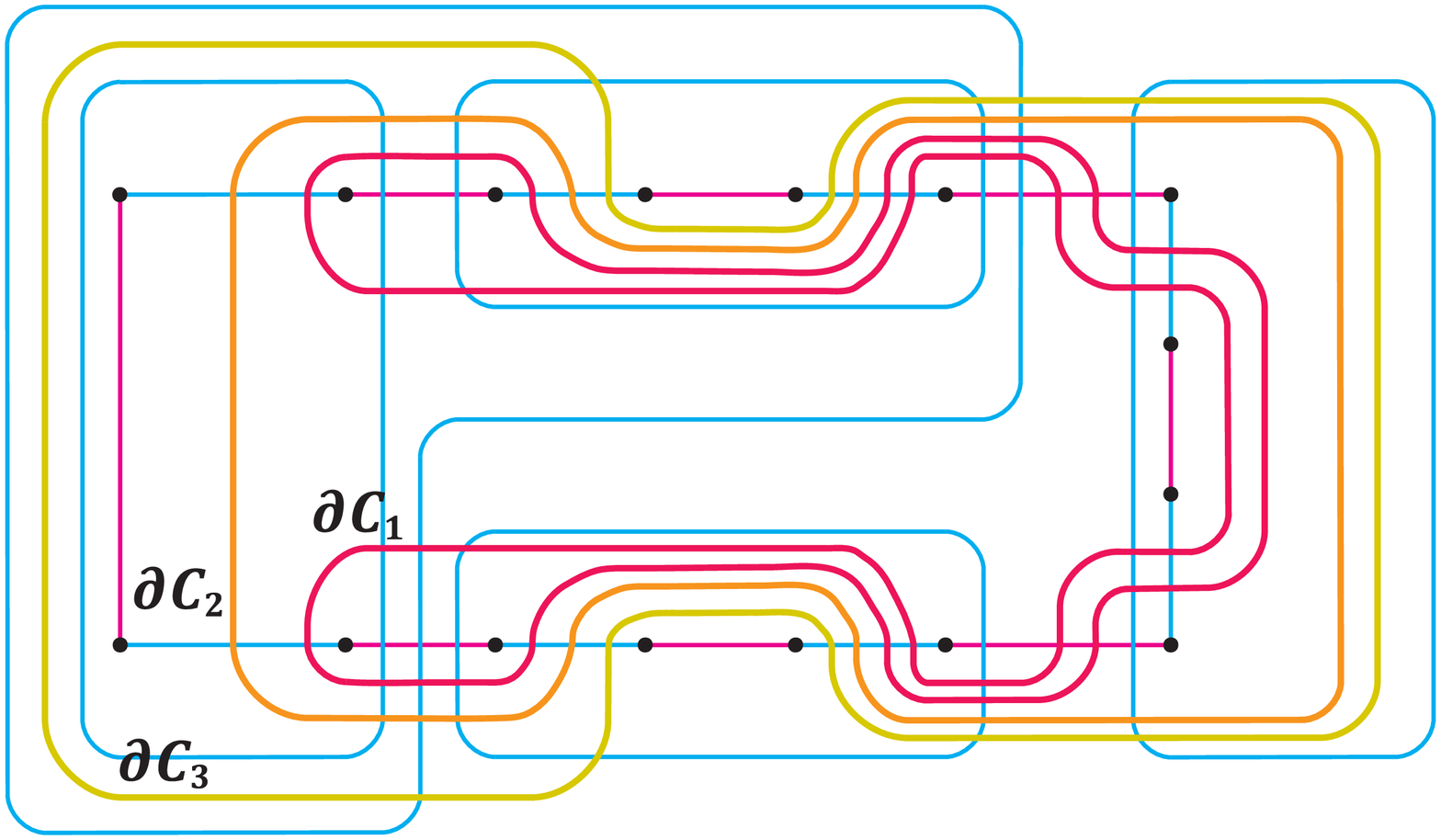}
\caption{$\partial C_1, \partial C_2$, and $\partial C_3$.}\label{fig3}
\end{center}
\end{figure}

\begin{figure}[ht!]
\begin{center}
\includegraphics[width=12cm]{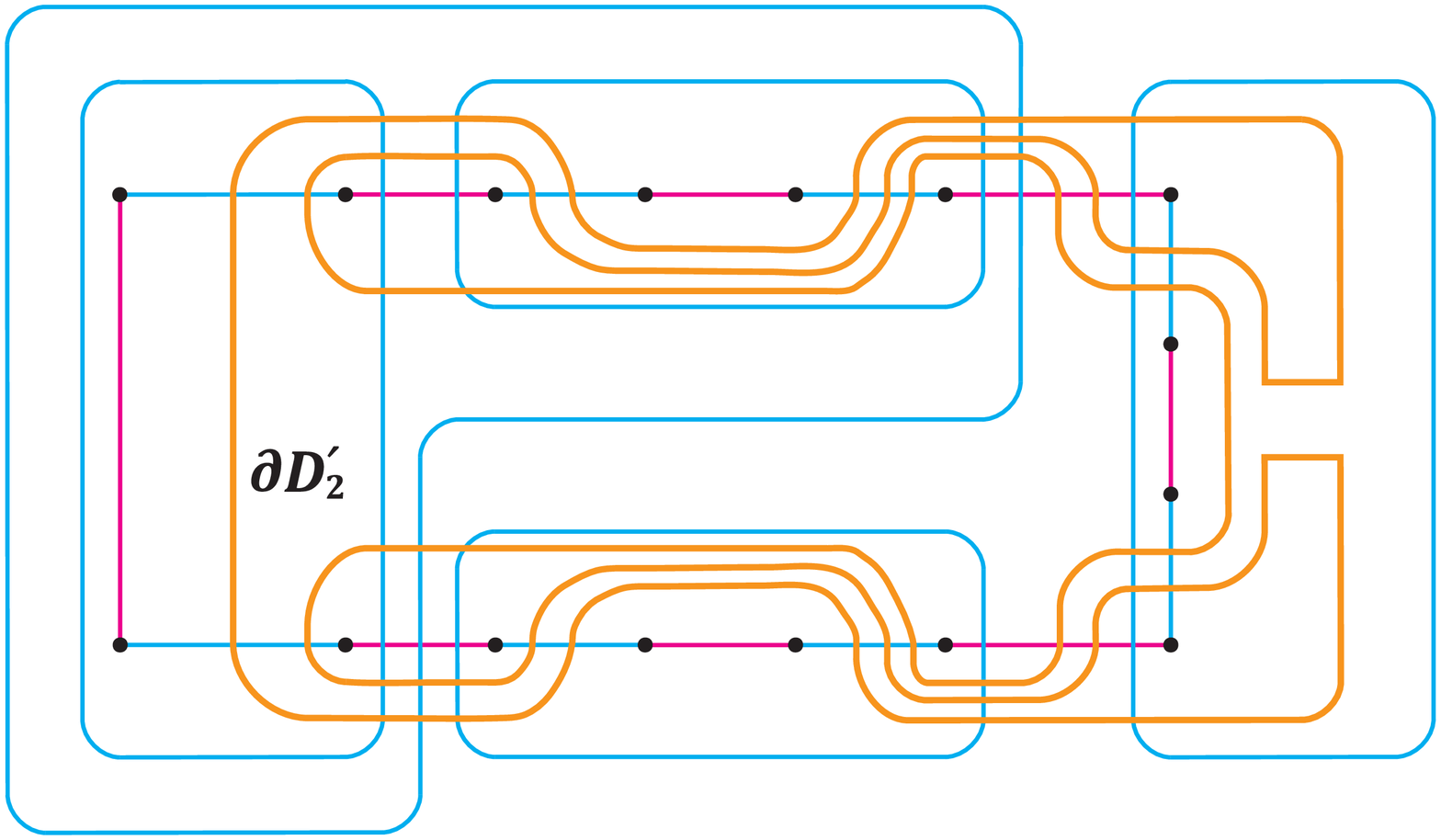}
\caption{The disk $D'_2$ is a band sum of $C_1$ and $C_2$.}\label{fig4}
\end{center}
\end{figure}

\begin{lemma}\label{lem2}
The disk $D'_2$ is not a weak reducing disk.
\end{lemma}

\begin{proof}
For any $P_i$ and any pair of two distinct boundary components of $P_i$,
there is a subarc of $d'_2$ in $P_i$ connecting the two components.
Suppose $F$ is a compressing disk in $W - K$ that is disjoint from $D'_2$.
Since $d'_2$ intersects $x$ minimally up to isotopy, we may assume that $F$ intersects $X$ minimally.
The disk $F$ cannot be isotopic to a cap over $a_{2i}$ and $E_i$ for any $i$.
So $F \cap X \ne \emptyset$.
Consider an outermost disk $F_0$ in $F$ cut off by an outermost arc of $F \cap X$.
The disk $F_0$ intersects some $P_i$ in a properly embedded arc
with two endpoints in the same component of $\partial P_i$, which is called a {\em {wave}}.
However, as mentioned at the beginning of the proof,
subarcs of $d'_2$ are an obstruction for such a wave, a contradiction.
\end{proof}

The disk $D''_2$ is a band sum of $C_1$ and $C_3$ as in Figure \ref{fig5}.
Let $d''_2 = \partial D''_2$.
It is essential in $S - K$.

\begin{figure}[ht!]
\begin{center}
\includegraphics[width=12cm]{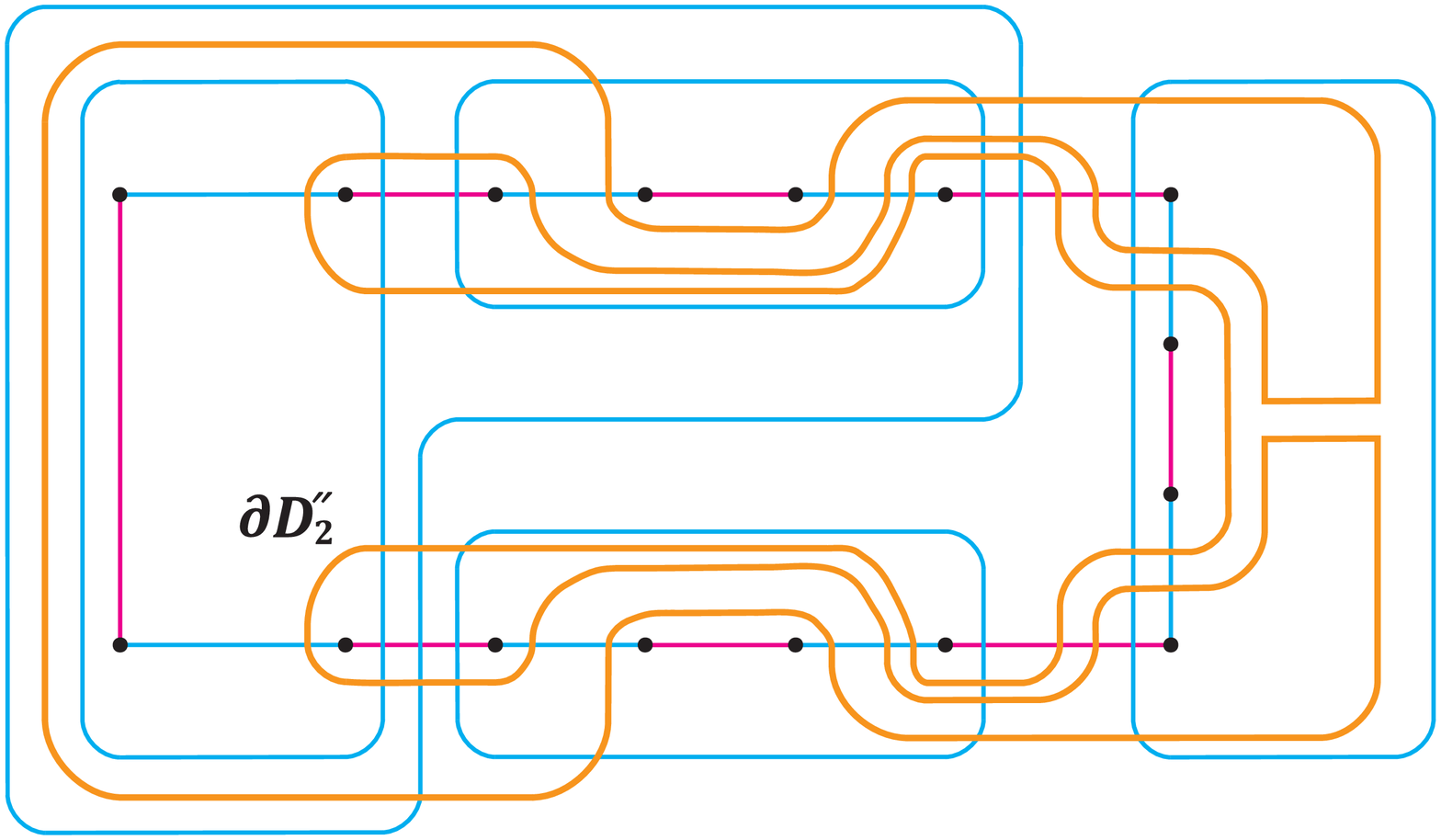}
\caption{The disk $D''_2$ is a band sum of $C_1$ and $C_3$.}\label{fig5}
\end{center}
\end{figure}

\begin{lemma}\label{lem3}
The disk $D''_2$ is not a weak reducing disk.
\end{lemma}

\begin{proof}
Suppose that $F$ is a compressing disk in $W - K$ disjoint from $D''_2$.
As in the proof of Lemma \ref{lem2}, we may assume that $F$ intersects $X$ minimally.
The disk $F$ cannot be isotopic to a cap over $a_{2i}$ and $E_i$ for any $i$,
so $F \cap X \ne \emptyset$.

\begin{claim}\label{claim1}
$\partial F$ intersects some $a_{2i}$, i.e.
the word $w$ of $\partial F$ is not an empty word before cancellation.
\end{claim}

\begin{proof}
Consider an outermost disk $F_0$ in $F$ cut off by an outermost arc of $F \cap X$.
It makes a wave.
A wave disjoint from $d''_2$ is possible only in $P_1$.
There are two possibilities for extending one end of a wave along $\partial F$.
See Figure \ref{fig6}.
In any case, $\partial F$ intersects $a_4$ or $a_{14}$.
\end{proof}

\begin{figure}[ht!]
\begin{center}
\includegraphics[width=12cm]{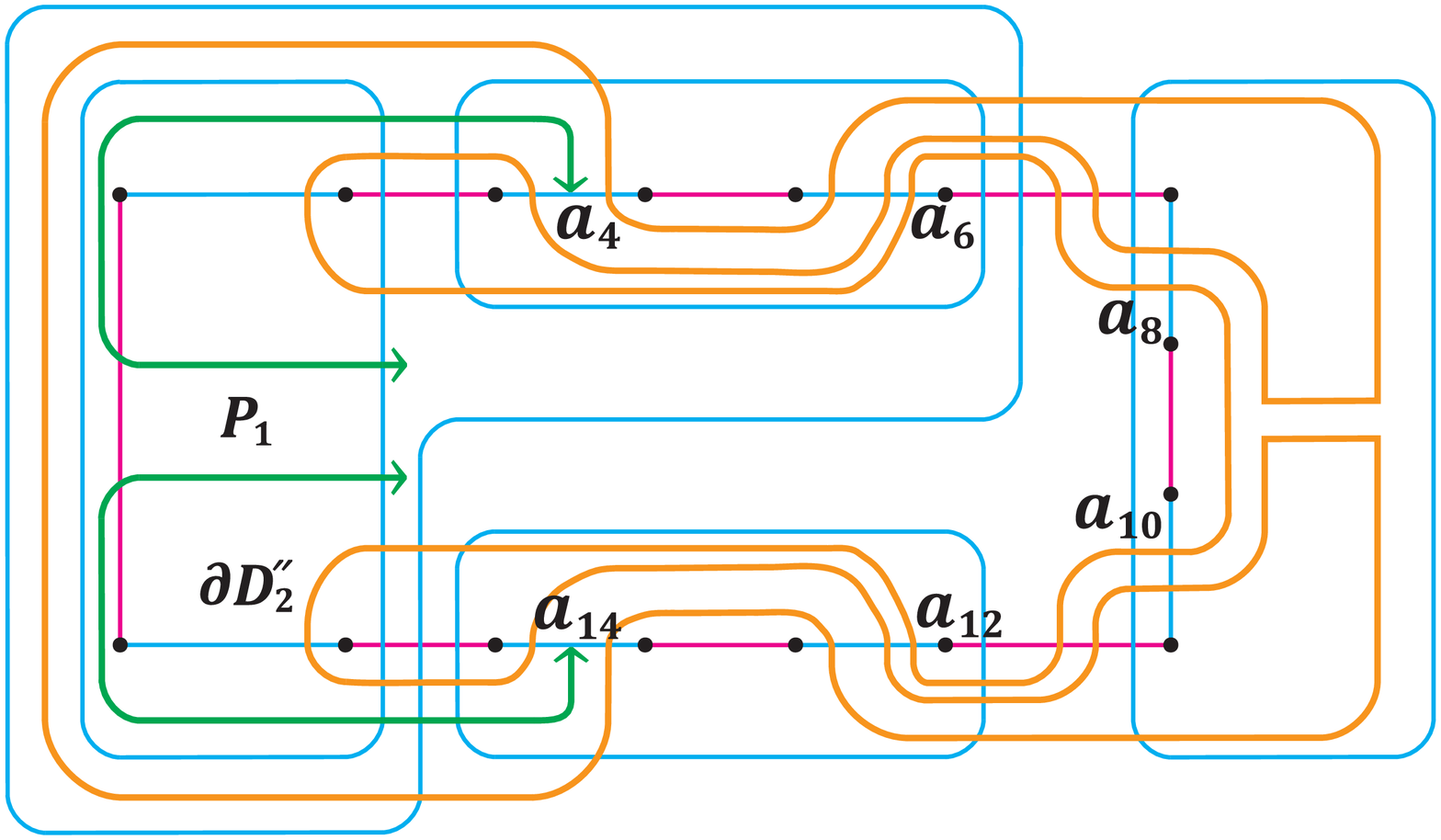}
\caption{Extending a wave.}\label{fig6}
\end{center}
\end{figure}

We use the same notation $a_{2i}$
for the generator of $\pi_1 (W - K)$ corresponding to the blue arc $a_{2i}$.
By Lemma \ref{lem1}, the nonempty word $w$ is freely reduced to an empty word.
Consider a cancellation of two adjacent generators, say $a_{2i}a^{-1}_{2i}$, in $w$.
Let $\beta$ be the subarc of $\partial F$ between the two points $p$ and $p'$
corresponding to $a_{2i}$ and $a^{-1}_{2i}$ respectively.
If $\beta$ does not intersect any $l_j$, then
$\beta$ is a wave with both endpoints in $a_{2i}$.
This cannot happen in our example, even in $P_1$.

Suppose that $\beta$ intersects $l_j$'s.
It necessarily intersects each $l_j$ at even number of points.
Let $p_1, \ldots, p_{2n}$ be the points of $\beta \cap (\bigcup^5_{j=1} l_j)$
appearing along $\beta$ in this order.
The two points $p_1$ and $p_{2n}$ belong to the same $l_j$.
Actually $p_k$ and $p_{2n+1-k}$ belong to the same $l_j$ for each $k$.

\begin{claim}\label{claim2}
$p_k$ and $p_{2n+1-k}$ belong to the same $l_j$ for each ($k = 1, \ldots, n$).
\end{claim}

\begin{proof}
Suppose that $m$ is the smallest index such that $p_m$ and $p_{2n+1-m}$ belong to the same $l_j$ but
$p_{m+1}$ and $p_{2n+1-(m+1)}$ belong to different $l_j$'s.
There are two cases.
The first case is that
neither $p_{m+1}$ nor $p_{2n+1-(m+1)}$ belongs to $l_j$ which contains $p_m$, and
the second case is that
one of $p_{m+1}$ or $p_{2n+1-(m+1)}$, say $p_{m+1}$, belongs to $l_j$ which contains $p_m$.
See Figure \ref{fig7}.
Note that a loop $l_j$ is separating in $S$ and
${\mathrm{int}} \, \beta$ does not intersect any blue ``arc".
So if a curve goes into a disk region bounded by $l_j$, then
there should be a wave in some pair of pants contained in the disk to come back to the outside.
Since there is only one pair of pants where wave is possible in our example,
one of the subarcs of $\beta$ cannot come back to outside.
\end{proof}

\begin{figure}[ht!]
\begin{center}
\includegraphics[width=14cm]{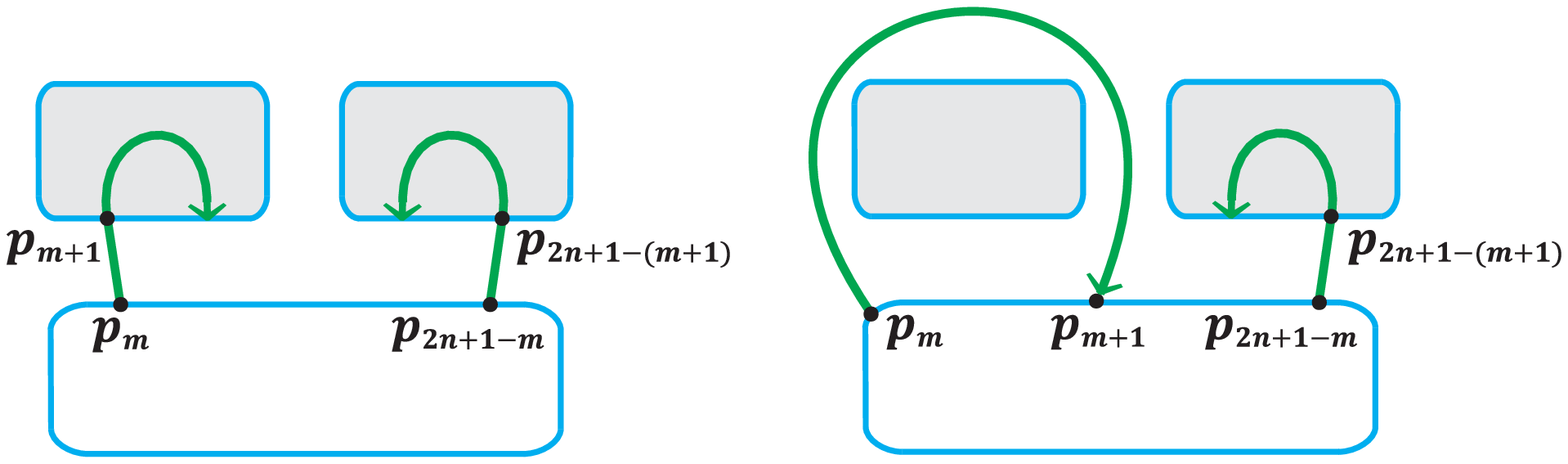}
\caption{The intersection of $\beta$ and $\bigcup l_j$.}\label{fig7}
\end{center}
\end{figure}

Hence $p_n$ and $p_{n+1}$ belong to the same $l_j$.
Then the subarc of $\beta$ between $p_n$ and $p_{n+1}$ is a wave.
As observed above, a wave is possible only in $P_1$.
Suppose we extend both ends of the wave along $\partial F$.
See Figure \ref{fig6}.
For one direction, $\partial F$ intersects $a_4$ or $a_{14}$ first among the blue arcs,
while for the other direction $\partial F$ cannot intersect $a_4$ or $a_{14}$ first.
(It intersects $a_6$ or $a_8$ or $a_{10}$ or $a_{12}$ first among the blue arcs.)
This contradicts that both endpoints of $\beta$ belong to the same $a_{2i}$.
\end{proof}

The disk $D_2$ is a band sum of $D'_2$ and $D''_2$ as in Figure \ref{fig8}.
The green curve in Figure \ref{fig8} bounds a compressing disk in $W - K$,
so $D_2$ is a weak reducing disk.
Hence both $D_1$ and $D_2$ are weak reducing disks.

\begin{figure}[ht!]
\begin{center}
\includegraphics[width=13cm]{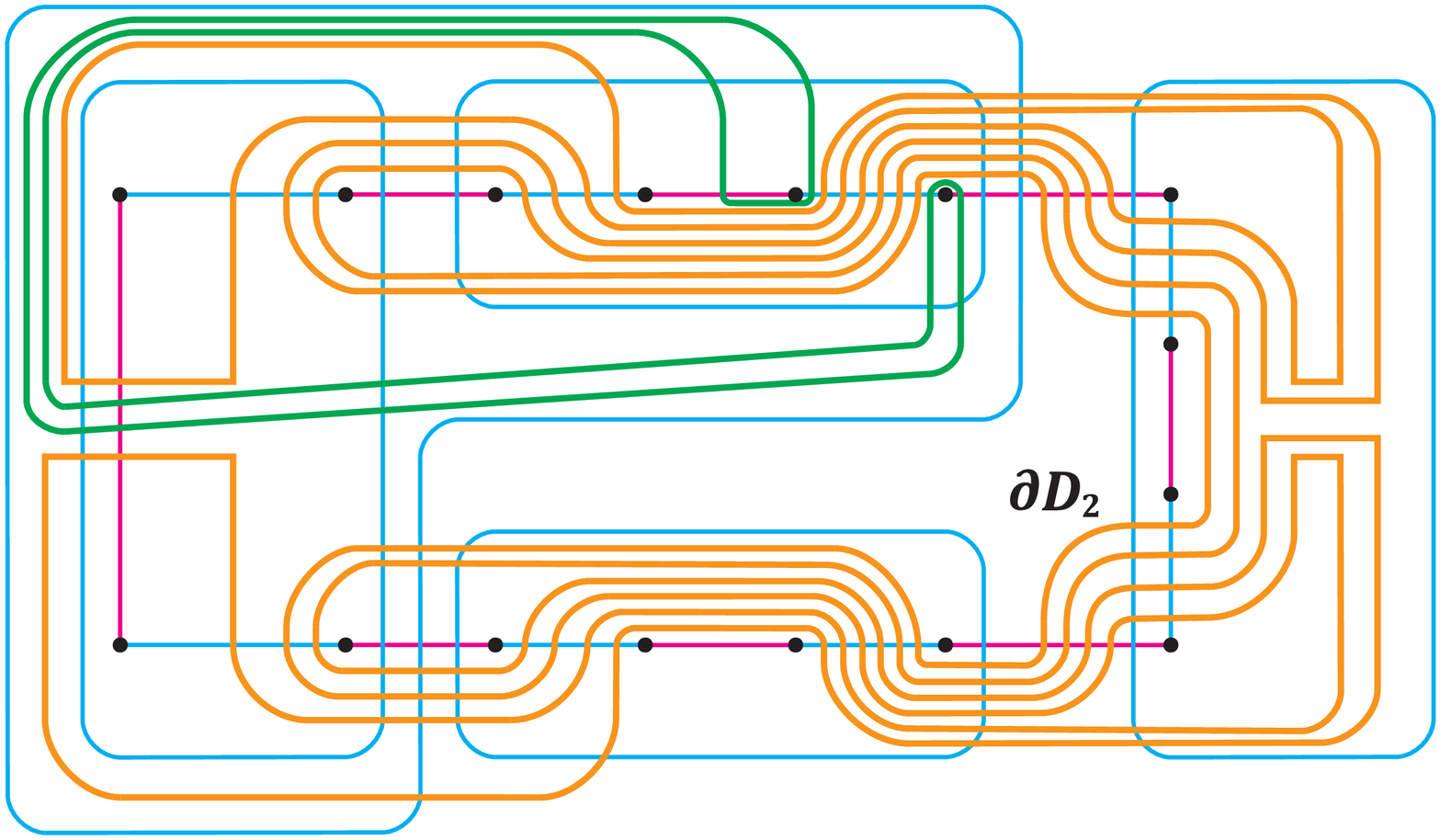}
\caption{$D_2$ is a weak reducing disk.}\label{fig8}
\end{center}
\end{figure}

The intersection $D_1 \cap D_2$ consists of two arcs $\alpha$ and $\beta$.
The arcs $\alpha$ and $\beta$ cut off outermost disks, say $G_1$ and $G_2$, from $D_1$ respectively.
Similarly, $\alpha$ and $\beta$ cut off outermost disks, say $H_1$ and $H_2$, from $D_2$ respectively.
See Figure \ref{fig9}.
A surgery of $D_1$ along $H_1$ results in two disks isotopic to $D'_2$ and $D''_2$ respectively, and
a surgery of $D_1$ along $H_2$ also results in two disks isotopic to $D'_2$ and $D''_2$ respectively.
A surgery of $D_2$ along $G_1$ (and $G_2$) also results in two disks
isotopic to $D'_2$ and $D''_2$ respectively.
This completes the proof of Theorem \ref{thm1}.

\begin{figure}[ht!]
\begin{center}
\includegraphics[width=13cm]{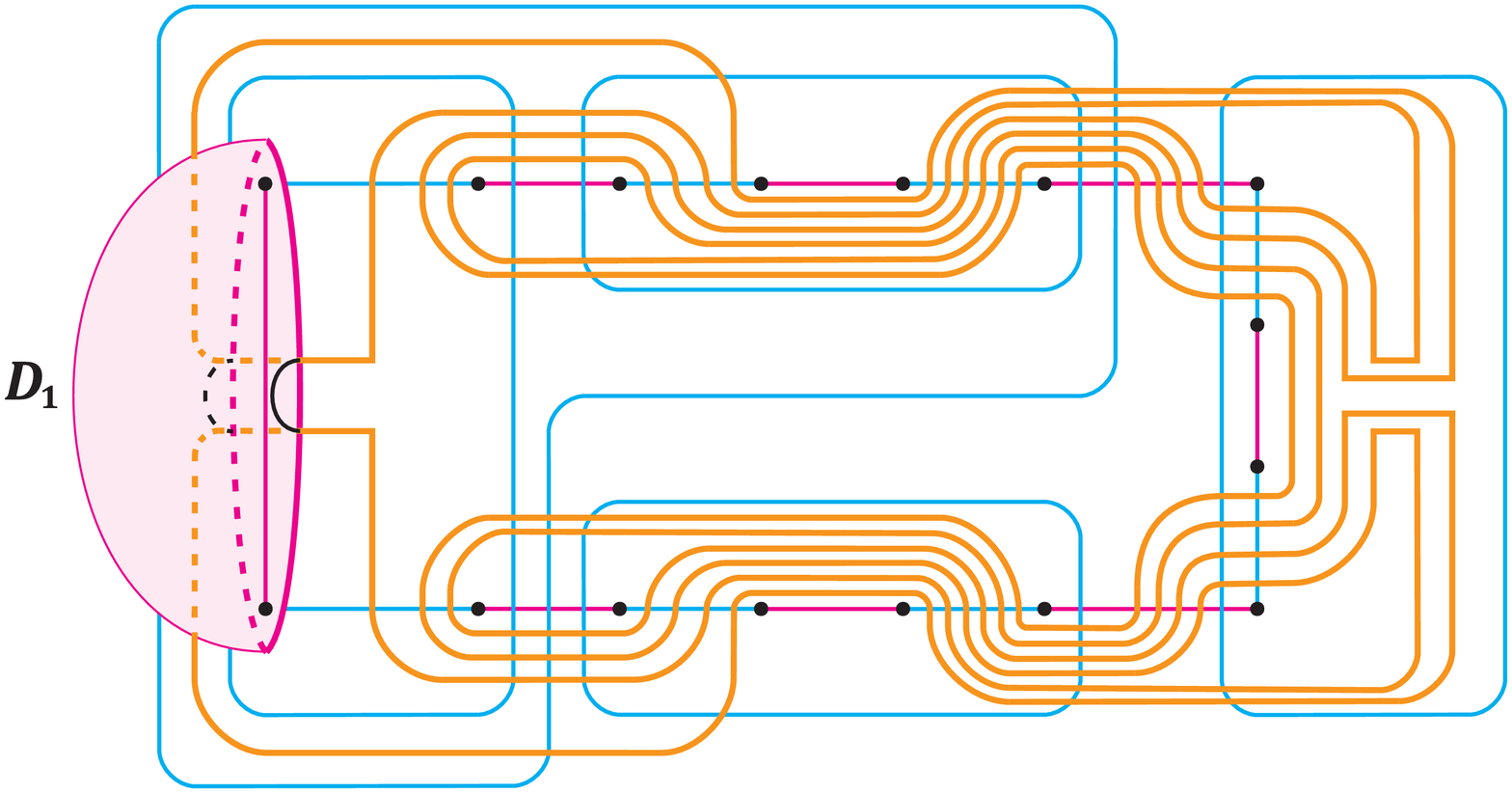}
\caption{A disk surgery does not yield any weak reducing disk.}\label{fig9}
\end{center}
\end{figure}

%{\bf Acknowledgments}

%This work was done while the author was visiting Osaka City University.
%The author is grateful to Taizo Kanenobu for his hospitality.

%This work was supported by the Basic Science Research Program
%through the National Research Foundation of Korea
%funded by the Ministry of Education [NRF-2015R1D1A1A01056953]. \\
%This paper was supported by research funds of Chonbuk National University in 2017.

\end{document}